\documentclass[journal]{IEEEtran}
\usepackage{mathrsfs}
\usepackage{amsmath}
\usepackage{graphicx}
\usepackage{amssymb}
\usepackage{enumerate}
\usepackage{longtable,tabularx,float}
\usepackage{cite}
\usepackage{mathcomp}
\usepackage{supertabular}
\usepackage{stmaryrd}
\usepackage{color}
\usepackage{url}
\newtheorem{corollary}{Corollary}[section]
\newtheorem{definition}{Definition}[section]
\newtheorem{lemma}{Lemma}[section]
\newtheorem{theorem}{Theorem}[section]

\newtheorem{example}{Example}[section]

\newcommand{\bbZ}{{\mathbb Z}}
\newcommand{\bbF}{{\mathbb F}}

\begin{document}

\title{Improved Constructions of Frameproof Codes}

\author{Yeow~Meng~Chee,~{\it Senior~Member,~IEEE},~and~Xiande~Zhang
        \thanks{The research of Y. M. Chee and X. Zhang was supported in part by the National Research
Foundation of Singapore under Research Grant NRF-CRP2-2007-03. The research of Y.
M. Chee was also supported in part by the Nanyang Technological
University under Research Grant M58110040.}
        \thanks{Y. M. Chee ({\tt ymchee@ntu.edu.sg}) 
        is with Division of Mathematical Sciences, School of
Physical and Mathematical Sciences, Nanyang Technological
University, Singapore 637371.} 

\thanks{X. Zhang ({\tt xdzhangzju@163.com}) is with School of Mathematical Sciences, Monash University,
 VIC 3800, Australia. This work was done while the author was with the Division of Mathematical
Sciences, School of Physical and Mathematical Sciences, Nanyang
Technological University, 21 Nanyang Link, Singapore 637371.}
}
\maketitle

\begin{abstract}
\boldmath Frameproof codes are used to preserve the security in
 the context of coalition when fingerprinting digital data. Let $M_{c,l}(q)$ be
the largest cardinality of a $q$-ary $c$-frameproof code of length
$l$ and $R_{c,l}=\lim_{q\rightarrow \infty}M_{c,l}(q)/q^{\lceil
l/c\rceil}$. It has been determined by Blackburn that $R_{c,l}=1$ when $l\equiv
1\ (\bmod\ c)$,  $R_{c,l}=2$ when $c=2$ and $l$ is even, and $R_{3,5}=\frac{5}{3}$. In this paper, we give a recursive construction for
$c$-frameproof codes of length $l$ with respect to the alphabet
size $q$. As  applications of this construction, we establish the existence results for $q$-ary
$c$-frameproof codes of length $c+2$ and size
$\frac{c+2}{c}(q-1)^2+1$
 for all odd $q$ when $c=2$ and for all $q\equiv 4\pmod{6}$ when
 $c=3$. Furthermore,  we show
that $R_{c,c+2}=(c+2)/c$ meeting the upper bound given by Blackburn, for all integers $c$ such that $c+1$
is a prime power.

\end{abstract}

\begin{IEEEkeywords}
\boldmath Fingerprinting, frameproof codes,  orthogonal array.
\end{IEEEkeywords}

\section{Introduction}

\IEEEPARstart{F}{rameproof} codes  were first introduced by Boneh and
Shaw \cite{Boneh1998IEEE} in 1998  to protect copyrighted materials.
When a distributor wants
to sell copies of a  digital product, he randomly chooses $l$
fixed positions in the digital data. For each copy, he marks each
position with one of $p$ different states. Such a collection of marked positions
in each copy is known as a {\it fingerprint}, which
can be thought as a codeword of length $l$ over an alphabet $F$ of
size $q$.
 The users don't know the positions and  states embedded in the data, so they cannot remove them.
However, in the context of collusion, some
users can share and compare their copies, and they can easily
discover some or perhaps all marked positions and create illegal copies.  A set of fingerprints is called to be {\it $c$-frameproof}
if any coalition of at most $c$ users can not frame another
 user not in the coalition.
\subsection{Related Objects}
The study of related objects to frameproof codes in the literature goes back to 1960s, as R\'enyi first introduced the concept of a separating system in his papers concerning certain information-theoretic problems \cite{renyi61a,renyi61b,renyi61c,renyi62a}. After that, the concept was defined again in cryptography several decades later, under different scenarios and purposes.
Besides the frameproof codes suggested by Boneh and
Shaw \cite{Boneh1998IEEE}, variants of such codes have become objects of study by many researchers. For instance,
\begin{enumerate}
\item[$\centerdot$]{\it secure frameproof codes} (SFP)
\cite{StinsonJSPS2000} are defined to demand that no coalition of
at most $c$  users can frame another disjoint coalition
of at most $c$ users; \item[$\centerdot$]Codes with {\it
identifiable parent property} (IPP) \cite{HollmannJCTA1998,DongSiam2004,NogaCPC2004,BlackburnJCTA2003}
require that no coalition of at most $c$  users can
produce a copy that cannot be traced back to at least one member
of the coalition; \item[$\centerdot$]{\it Traceability codes} (TA)
\cite{StinsonSiam1998,Chor1994AC,Fiat2001JC,StaddonIEEE2001} have much stronger identifiable parent
property which allows an efficient (i.e., linear-time in the size of the code) algorithm
to determine one member of the coalition.
\end{enumerate}

 The intimate relations
among such kinds of codes and connections with other combinatorial
objects, such as certain types of separating hash families,
cover-free families and combinatorial group testings were described
in
\cite{Chor1994AC,Fiat2001JC,StinsonSiam1998,StaddonIEEE2001,ColbournIEEE2010}.
These have motivated much research investigating the constructions
and bounds of these  codes, and of related objects, see for
example
\cite{DongSiam2004,BlackburnJCTA2010,NogaCPC2004,BlackburnJCTA2003,WangJCTA2001,BazrafshanJCTA2011,HollmannJCTA1998,BlackburnJCTA2008,StinsonJCTA2008,XingIEEE2002,CohenEL2000,Blackburn2003SIAM}.

\subsection{Preliminaries}

In this paper, we mainly investigate the upper bounds and constructions of
frameproof codes. The definition we use  was explicitly given by Fiat and Tassa
\cite{Fiat2001JC}, who credited Chor, Fiat, and Naor
\cite{Chor1994AC} with its first use.

Let $F$ be a finite set of cardinality $q$ and $l$ be a positive
integer. The set $\{1,\ldots,l\}$ is denoted by $[l]$. For a
$q$-ary word $x\in F^l$ and an integer $i \in[l]$ we write $x_i$
for the $i$th component of $x$. Let $P \subset F^l$ be a set of
words of length $l$. The set of {\it descendants} of $P$,
$desc(P)$, is the set of all words $x\in F^l$ such that for all $i
\in[l]$, there exists $y\in P$ satisfying $x_i = y_i$, i.e.,
\begin{equation*}
desc(P)=\{x\in F^l: x_i\in \{y_i:y \in P\}, i\in [l]\}.
\end{equation*}
 Let $c$ be an integer such that $c \geq2$. A {\it $c$-frameproof
code} is a subset $C \subset F^l$ such that for all $P \subset C$ with
$|P|\leq c$, we have that $desc(P) \cap C = P$.

Let $M_{c,l}(q)$ be the largest cardinality of a $q$-ary
$c$-frameproof code of length $l$. Staddon, Stinson and Wei \cite{StaddonIEEE2001} proved an upper bound
for $M_{c,l}(q)$, $q\geq 2$, which is given as follows:
\begin{equation*}
M_{c,l}(q)\leq c\big(q^{\lceil
l/c\rceil}-1\big).
\end{equation*}
The exact value of $M_{c,l}(q)$ is not known except for the
trivial case, i.e., when $l\leq c$ and $q\geq 2$,
$M_{c,l}(q)=l(q-1)$ shown by Blackburn \cite{Blackburn2003SIAM}.
So the more interesting and difficult case is  when $l
>c$.
 In \cite{Blackburn2003SIAM}, Blackburn also established an asymptotic upper bound for
$M_{c,l}(q)$, which is restated as follows.
\vskip 10pt
\begin{theorem}\cite{Blackburn2003SIAM}
\label{lowerbound}  Let $c$,   $ l$ and $q$  be positive integers
greater than $1$. Let $t \in[c]$ be an integer such that $t\equiv l\pmod{c}$.
Then
\begin{equation*}
M_{c,l}(q)\leq\big(\frac{l}{l-(t-1)\lceil l/c\rceil}\big)q^{\lceil
l/c\rceil}+O(q^{\lceil l/c\rceil-1}).
\end{equation*}
\end{theorem}

 Let $R_{c,l}(q)=M_{c,l}(q)/q^{\lceil l/c\rceil}$ and $R_{c,l}=\lim_{q\rightarrow
\infty}R_{c,l}(q)$. Then it is easy to show
the following result by Theorem~\ref{lowerbound}.
\vskip 10pt
\begin{corollary}
\label{lowerbound1} Let $c$ and $l$  be positive integers
greater than $1$. Let $t \in[c]$ be an integer such that $t\equiv l\pmod{c}$.
Then
\begin{equation*}
R_{c,l}\leq\frac{l}{l-(t-1)\lceil l/c\rceil}.
\end{equation*}
\end{corollary}

 When $l> c$, Blackburn
\cite{Blackburn2003SIAM} showed that $R_{c,l}=1$ when $l\equiv
1\ (\bmod\ c)$, and $R_{c,l}=2$ when $c=2$ and $l$ is even. The
next most tempting case is when $t=2$, i.e., $l\equiv 2\pmod{c}$.
Blackburn asked in \cite[Section 8]{Blackburn2003SIAM} the
following question: Is there a $q$-ary
 $c$-frameproof code of length $l$ with cardinality approximately $l/(l-\lceil
 l/c\rceil) q^{\lceil l/c\rceil}$ when $l\equiv 2\pmod{c}$? In fact, the answer is yes when $l=5$ and $c=3$, which was  proved in \cite[Construction 4]{Blackburn2003SIAM}   by constructing a $3$-frameproof code of
length $5$ of sufficiently large cardinality.

 Inspired by this question, we pursue  the exact values for $R_{c,l}$ with $l=c+2$ in the following sections by constructing
 $c$-frameproof codes with cardinality asymptotically meeting the upper bound in Theorem~\ref{lowerbound}. The paper is organized as follows.
In Section II, we present a general recursive
 construction for $c$-frameproof codes of length $l$ with respect to the alphabet size $q$ by introducing the
  definition of Property $P(t)$ for a frameproof code. As  applications of this method, we establish the existence results of $q$-ary $c$-frameproof codes of length $c+2$ and size $\frac{c+2}{c}(q-1)^2+1$
 for all odd $q$ when $c=2$ and for all $q\equiv 4\pmod{6}$ when $c=3$ in Section III.  In Section IV, we apply the method to  the frameproof
  codes obtained from orthogonal arrays to prove that the
 upper bound for $R_{c,l}$ in Corollary~\ref{lowerbound1} can be achieved for all $c\geq2$ and $l=c+2$ when $c+1$ is a prime
 power. Finally, we conclude our paper in Section V.

\section{A General Recursive Construction}
This section serves to describe a general recursive construction
for $c$-frameproof codes. First, we introduce the definition of
Property $P(t)$ for a code, where $t$ is a positive integer.

\vskip 10pt
\begin{definition}\label{defpt}
Let $C$ be an $s$-ary $c$-frameproof code
 of length $l$ over an alphabet $S$ of size $s$.  $C$ is said to satisfy
{\it Property $P(t)$} if there exists a special element say
$\infty\in S$, such
 that each codeword contains at most $t-1$ $\infty$'s and is uniquely determined by specifying $t$ of its
components that are not equal to
 $\infty$.
\end{definition}
\vskip 10pt

Now suppose $C$ is an $s$-ary $c$-frameproof code
 of length $l$ over $S$  satisfying  Property $P(t)$ with a special element $\infty$, $l\geq 2t-1$.
For convenience,  let $T=S\setminus\{\infty\}$. Suppose $C$ has
cardinality $M$. Denote the codewords of $C$ by $B_i$ with $i\in
[M]$. By Definition~\ref{defpt}, there are at most $t-1$
components with $\infty$ of $B_i$ for each $i\in [M]$.
Furthermore, a codeword $B_i$ is uniquely determined by specifying
$t$ of its components  that are not equal to
 $\infty$.

Let $m$ be a prime power such that $m\geq l-1$ and $\bbF_m$ be the finite field of order $m$. Let
$\{\alpha_1,\alpha_2,\ldots,\alpha_{l}\}$ be a set of $l$ distinct
elements in the alphabet $\bbF_m\cup \{\infty\}$. For each polynomial $f\in
\bbF_m[X]$, let $f_{\infty}$ denote the coefficient of $X^{t -1}$
in $f$. For each $B_i\in C$, denote $B_i=(b_1,b_2,\ldots,b_l)$.
Let $Y_i$ be a set of words of length $l$ over $\bbF_m\cup
\{\infty\}$, such that each word $y=(y_1,\ldots,y_l)\in Y_i$ is
defined by
\begin{equation*}
y_j=\begin{cases}\infty, &\text{if $b_j=\infty$;}\\
f_{\infty}, &\text{if $b_j\neq\infty$ and $\alpha_j=\infty$;}\\f(\alpha_j),&
\text{otherwise,}
\end{cases}
\end{equation*}
with $j\in [l]$, where $f$ runs over $\bbF_m[X]$ with $\deg f\leq
t-1$. So each word $y\in Y_i$ is uniquely determined by specifying
$t$ components  that are not equal to $\infty$. Moreover, since $l\geq
2t-1$, i.e., $l-(t-1)\geq t$, all the words of $Y_i$ are distinct.
Hence each set $Y_i$ has cardinality $m^{t}$.

Now for each $i\in [M]$, define a set $C_i$ of words of length $l$
over $(T\times\bbF_m)\cup \{(\infty,\infty)\}$ by
\begin{equation*}
\begin{split}
C_i=\{&((b_1,y_1),(b_2,y_2),\ldots,(b_l,y_l)):
B_i=(b_1,b_2,\ldots,b_l)\\
& \text{ and }(y_1,y_2,\ldots,y_l)\in Y_i\}.
\end{split}
\end{equation*}
Let $C'=\cup_{i=1}^MC_i$. It is clear that all $C_i$ are disjoint,
thus $|C'|=Mm^t$. The following lemma proves that $C'$ is also a
$c$-frameproof code.

\vskip 10pt
\begin{lemma}
\label{p2code} Let $m$ be a prime power and $l,s,t$ be
positive integers such that $m\geq l-1$ and $2t-1\leq l$. Define
$q=(s-1)m+1$. Suppose that $c\geq t$ is an integer such that $l=c(t-1)+r$
for some $r\in \{t,t+1,\ldots,c\}$. If there exists an $s$-ary
length $l$ $c$-frameproof code of cardinality $M$ satisfying
Property $P(t)$, then there exists a $q$-ary length $l$
$c$-frameproof code of cardinality $Mm^t$.
\end{lemma}
\begin{proof} Using the same notations and construction as above, it remains to show that $C'=\cup_{i=1}^MC_i$ is a $c$-frameproof code of length $l$ over
$(T\times\bbF_m)\cup \{(\infty,\infty)\}$.

For each word $x=((b_1,y_1),(b_2,y_2),\ldots,(b_l,y_l))\in C'$, let $\pi_k(x)$ be the word by mapping
each element to its $k$th coordinate, $k=1,2$, i.e., $\pi_1(x)=(b_1,b_2,\ldots,b_l)$ and $\pi_2(x)=(y_1,y_2,\ldots,y_l)$. Suppose $x\in C'$
and let $P\subset C'$ be such that $|P|\leq c$ and $x\in desc(P)$.
We will show that $x\in P$. Since $|P|\leq c$ and $r\geq t$, there
exists $y\in P$ that agrees with $x$ in $t$ or more components that are not equal to $(\infty,\infty)$. We aim to show $x=y$.

 Since $x\in C'$, there exists $i$ such that
$x\in C_i$. Then $\pi_1(x)=B_i$. Since $\pi_1(x)$ and $\pi_1(y)$
agree in $t$ or more components that are not equal to $\infty$,
$\pi_1(x)=\pi_1(y)=B_i$. That is $y\in C_i$. Thus $\pi_2(x)$ and
$\pi_2(y)$ are both in $Y_i$. Since $\pi_2(x)$ and $\pi_2(y)$
agree in $t$ or more components that are not equal to $\infty$,
$\pi_2(x)=\pi_2(y)$. Hence  $x=y\in P$ as required.
\end{proof}
\vskip 10pt

Let $\bbZ_p$ be the ring of integers modulo $p$. Here are two examples as applications of
Lemma~\ref{p2code}.

\vskip 10pt
\begin{example}
\label{c=2,l=4} Let $S=\{\infty\}\cup \bbZ_{2}$. Define four sets
$X_1$, $X_2$, $X_3$ and $X_4$ of words of length $4$ over $S$ as
follows:
\begin{equation*}
\begin{split}
&X_1=\{(\infty,i,i,i):i\in \bbZ_2\},\\ &X_2=\{(i,\infty,i,i+1):i\in
\bbZ_2\},\\&X_3=\{(i,i+1,\infty,i):i\in
\bbZ_2\},\\&X_4=\{(i,i,i+1,\infty):i\in \bbZ_2\}.
\end{split}
\end{equation*}
It is clear that the sets $X_i$ are pairwise disjoint and have
cardinality $2$. Let $C=\cup_{i=1}^4 X_i$, it is not difficult to
check that $C$ is a $3$-ary $2$-frameproof code of length $4$ over
$S$ with cardinality $8$. Furthermore, $C$ satisfies  Property
$P(2)$. Let $m\geq 3$ be any prime power and $q=2m+1$. By applying
Lemma~\ref{p2code}, there exists a $q$-ary $2$-frameproof code of
length $4$  of cardinality $8m^2=2(q-1)^2$.
\end{example}
\vskip 10pt

 Note: Example~\ref{c=2,l=4} shows that $M_{2,4}(q)\geq
2(q-1)^2$ for each $q=2m+1$ with $m\geq 3$ a prime power. In
\cite[Construction 3]{Blackburn2003SIAM}, Blackburn constructed a
$q$-ary $2$-frameproof code of length $4$ of cardinality
$2(q-1)^2(1-1/(2\sqrt{q-1}))$, where $q=m^2+1$ and $m\geq 5$ is a
prime power. In this case, Example~\ref{c=2,l=4} constructs
$2$-frameproof codes of length $4$ with bigger size for a more dense family of parameters $q$.

 \vskip 10pt
\begin{example}
\label{c=3,l=5}This is from \cite[Construction
$4$]{Blackburn2003SIAM}. Let $S=\{\infty\}\cup \bbZ_{3}$. Define five
sets $X_1$, $X_2$, $X_3$, $X_4$ and $X_5$ of words of length $5$
over $S$ as follows:
\begin{equation*}
\begin{split}
&X_1=\{(\infty,i,i,i,i):i\in \bbZ_3\},\\
&X_2=\{(i,\infty,i,i+1,i+2):i\in
\bbZ_3\},\\&X_3=\{(i,i,\infty,i+2,i+1):i\in
\bbZ_3\},\\&X_4=\{(i,i+1,i+2,\infty,i):i\in \bbZ_3\},\\
&X_5=\{(i,i+2,i+1,i,\infty):i\in \bbZ_3\}.
\end{split}
\end{equation*}
It is easy to see that the sets $X_i$ are pairwise disjoint and
have cardinality $3$. Let $C=\cup_{i=1}^5 X_i$, which forms a
$4$-ary $3$-frameproof code of length $5$ over $S$ with
cardinality $15$. Clearly $C$ satisfies Property $P(2)$. Let
$m\geq 4$ be a prime power and $q=3m+1$. By applying
Lemma~\ref{p2code}, there exists a $q$-ary $3$-frameproof code of
length $5$  of cardinality $15m^2=\frac{5}{3}(q-1)^2$.
\end{example}

\vskip 10pt

 Before the end of this section, we show that the resultant codes obtained from
Lemma~\ref{p2code} also satisfy Property $P(t)$, which means that
Lemma~\ref{p2code} can be applied recursively.

\vskip 10pt
\begin{lemma}
\label{p} Any frameproof code obtained from Lemma~\ref{p2code}
satisfies Property $P(t)$ with the same $t$ of the previous code.
Furthermore, the code is still $c$-frameproof after joining the
all $(\infty,\infty)$ codeword.
\end{lemma}
\begin{proof} We use the same notations as in Lemma~\ref{p2code}.
 The proof that $C'$ satisfies Property $P(t)$ with the special element $(\infty,\infty)$ is a straightforward verification by the construction and omitted.

Let $C_{\infty}$ be the all $(\infty,\infty)$ codeword. First, we
prove that $C_{\infty}$ is not in the descendant of any set
$P\subset C'$ with $|P|\leq c$. In fact, each codeword in $C'$
contains at most $t-1$ components with $(\infty,\infty)$. Hence
there are at most $c(t-1)$ $(\infty,\infty)$'s  contained in any
set $P\subset C'$ with $|P|\leq c$, but there are
$l=c(t-1)+r>c(t-1)$ $(\infty,\infty)$'s in $C_{\infty}$. Second,
suppose $x\in C'$ and let $P\subset C'$ be such that $|P|\leq c-1$
and $x\in desc(P\cup \{C_{\infty}\})$. We will show that $x\in P$.
Since $x$ has at least $(c-1)(t-1)+r$ components that are not equal to
$(\infty,\infty)$,
 there exists $y\in P$ that agrees with
$x$ in $t$ or more components that are not equal to $(\infty,\infty)$. By
the Property $P(t)$ of $C'$, $x=y$ as required. This completes the
proof.
\end{proof}
\vskip 10pt

\section{ $c=2$ and $3$}
In this section, we establish the existence of two infinite
families of $c$-frameproof codes of length $c+2$ with cardinality
$\frac{c+2}{c}(q-1)^2+1$ with $c=2$ and $3$.
\subsection{$c=2$ and $l=4$}
\vskip 10pt
\begin{lemma}
\label{c=2,l=4,m=5} There exists a $5$-ary $2$-frameproof code
with length $4$ of cardinality $32$ satisfying  Property $P(2)$.
\end{lemma}
\begin{proof} We will construct the $2$-frameproof code $C$ of length
$4$ over $(\bbZ_2\times \bbF_2)\cup \{\infty\}$. For each polynomial
$f\in \bbF_2[X]$, let $f_{\infty}$ denote the coefficient of $X$
in $f$. First, define four sets of words as follows:
\begin{equation*}
\begin{split}
X_1=\{&(\infty,(i,f(0)),(i,f(1)),(i,f_{\infty})):\\
& i\in \bbZ_2,f\in\bbF_2[X], deg f\leq 1 \},\\
X_2=\{&((i,f(0)),\infty,(i,f(1)),(i+1,f_{\infty})):\\
&i\in \bbZ_2,f\in\bbF_2[X], deg f\leq 1
\},\\X_3=\{&((i,f(0)),(i+1,f(1)),\infty,(i,f_{\infty})):\\
& i\in \bbZ_2,f\in\bbF_2[X], deg f\leq 1
\},\\X_4=\{&((i,f(0)),(i,f(1)),(i+1,f_{\infty}),\infty):\\
& i\in \bbZ_2,f\in\bbF_2[X], deg f\leq 1 \}.
\end{split}
\end{equation*}
It is clear that the set $X_i$ are pairwise disjoint and have
cardinality $8$. Let $C=\cup_{i=1}^4 X_i$, then $C$ is easily seen to be
a $5$-ary $2$-frameproof code of length $4$ with cardinality $32$ satisfying  Property $P(2)$.
\end{proof}
 \vskip 10pt

By applying Lemma~\ref{p2code}, we establish the following
existence result for $2$-frameproof codes.
 \vskip 10pt

\begin{theorem}
\label{c=2,l=4,m} There exists a $q$-ary $2$-frameproof code with
length $4$ of cardinality $2(q-1)^2+1$ for any odd $q>1$.
\end{theorem}
\begin{proof}  By Lemma~\ref{p}, it is sufficient to prove that for each odd $q>1$, there exists a $q$-ary $2$-frameproof code with
length $4$ of cardinality $2(q-1)^2$ satisfying Property $P(2)$.

For $q=3,5$, the conclusion is true by Example~\ref{c=2,l=4} and
Lemma~\ref{c=2,l=4,m=5}. Assume it is true for all odd integers
less than $2m+1$, $m\geq 3$, i.e., there exists a $q$-ary
$2$-frameproof code with length $4$ of cardinality $2(q-1)^2$
satisfying Property $P(2)$ for any odd $q<2m+1$. The proof proceeds by induction. If $m$ is a prime
power, then by Lemma~\ref{p2code} and Example~\ref{c=2,l=4}, such
a code exists for $q=2m+1$. If $m$ is not a prime power, write $m$ as
$m=p_1^{e_1}p_2^{e_2}\cdots p_s^{e_s}$.  There exists at least one
$i$, such that $p_i^{e_i}\geq 3$ is odd. Since
$2m/p_i^{e_i}+1<2m+1$ is odd, the code exists for
$q=2m/p_i^{e_i}+1$ by assumption. Then by Lemma~\ref{p2code}, the code exists
for $q=2m+1=((2m/p_i^{e_i}+1)-1)p_i^{e_i}+1$ as required.
\end{proof}
\vskip 10pt

\subsection{$c=3$ and $l=5$}
\vskip 10pt
\begin{lemma}
\label{c=3,l=5,m=10} There exists a $10$-ary $3$-frameproof code
with length $5$ of cardinality $135$ satisfying  Property $P(2)$.
\end{lemma}
\begin{proof}
 For each polynomial $f\in \bbF_3[X]$, let $f_{\infty}$ denote
the coefficient of $X$ in $f$. Now we define $C$ consisting of the
following five types of codewords over $(\bbZ_3\times \bbF_3)\cup
\{\infty\}$ with $i\in \bbZ_3,f\in\bbF_3[X]$ and $\deg f\leq 1$:
\begin{equation*}
\begin{split}
&(\infty,(i,f(0)),(i,f(1)),(i,f(2)),(i,f_{\infty})),\\
&((i,f(0)),\infty,(i,f(1)),(i+1,f(2)),(i+2,f_{\infty})),\\
&((i,f(0)),(i,f(1)),\infty,(i+2,f(2)),(i+1,f_{\infty})),\\
&((i,f(0)),(i+1,f(1)),(i+2,f(2)),\infty,(i,f_{\infty})),\\
&((i,f(0)),(i+2,f(1)),(i+1,f(2)),(i,f_{\infty}),\infty).
\end{split}
\end{equation*}
It is easy to check that $C$ is a $10$-ary $3$-frameproof code of
length $5$ with cardinality $135$ satisfying  Property $P(2)$.
\end{proof}
\vskip 10pt

Similar to the proof of Theorem~\ref{c=2,l=4,m}, we obtain the
following existence result for $3$-frameproof codes by induction.

\vskip 10pt
\begin{theorem}
\label{c=3,l=5,m} There exists a $q$-ary $3$-frameproof code with
length $5$ of cardinality $\frac{5}{3}(q-1)^2+1$ for any integer
$q\equiv 4\pmod{6}$.
\end{theorem}
\begin{proof} By Lemma~\ref{p}, it is sufficient to prove that for each $q\equiv 4\pmod{6}$,  there exists a $q$-ary $3$-frameproof code with
length $5$ of cardinality $\frac{5}{3}(q-1)^2$ satisfying Property
$P(2)$.

For $q=4,10$, the above statement is true by Example~\ref{c=3,l=5} and
Lemma~\ref{c=3,l=5,m=10}. Assume it is true for all integers
$q\equiv 4\pmod{6}$ less than $6m+4=3(2m+1)+1$, $m\geq 2$, i.e.,
there exists a $q$-ary $3$-frameproof code with length $5$ of
cardinality $\frac{5}{3}(q-1)^2$ satisfying Property $P(2)$ for
any integer $q\equiv 4\pmod{6}$ less than $3(2m+1)+1$. If $2m+1$
is a prime power, then by Lemma~\ref{p2code} and
Example~\ref{c=3,l=5}, such a code exists when $q=3(2m+1)+1$. If $2m+1$ is
not a prime power, assume $2m+1=p_1^{e_1}p_2^{e_2}\cdots
p_s^{e_s}$. Since $2m+1\geq 5$ is odd, there exists at least one
$i$, such that $p_i^{e_i}\geq 5$ is odd. Since
$3(2m+1)/p_i^{e_i}+1\equiv 4\pmod{6}$ is less than $3(2m+1)+1$, the code exists for $q=3(2m+1)/p_i^{e_i}+1$ by assumption. Then by Lemma~\ref{p2code},
the conclusion is true for $q=3(2m+1)+1=((3(2m+1)/p_i^{e_i}+1)-1)p_i^{e_i}+1$
as required.
\end{proof}
\vskip 10pt

\section{Determination of $R_{c,c+2}$}
Having demonstrated in Section III, that $R_{c,c+2}=\frac{c+2}{c}$ when $c=2,3$, we now pursue the determination of $R_{c,c+2}$ for general $c$. We begin by introducing the definition of orthogonal arrays.

An {\em orthogonal array} of size $N$, with $k$ constraints (or of
degree $k$), $s$ levels (or of order $s$), and strength $t$,
denoted by $OA(N,k,s,t)$, is a $k \times N$ array with entries
from a set of $s \geq 2$ symbols, having the property that in
every $t\times N$ submatrix, every $t\times1$ column vector
appears the same number $\lambda = \frac{N}{s^t}$ of times. The
parameter $\lambda$ is the index of the orthogonal array. An
$OA(N,k,s,t)$ is also denoted by $OA_{\lambda}(t,k,s)$. If $t$ is
omitted, it is understood to be $2$. If $\lambda$ is omitted, it
is understood to be $1$.

 Orthogonal arrays are well known  used to give codes of high minimum distance.
It was proved in \cite{StaddonIEEE2001,Chor1994AC} that codes with high minimum distance are frameproof codes with some parameters.
To make the paper self-contained, we prove the following result from
orthogonal arrays.

\vskip 10pt
\begin{lemma}
\label{oa2code} If there exists an $OA(t,l,s)$, then there exists
an $s$-ary length $l$ $c$-frameproof code of cardinality $s^t$,
where $c$ is any integer such that $l>c(t-1)$.
\end{lemma}
\begin{proof}  Suppose the given
$OA(t,l,s)$ is an $l\times s^t$ array with entries from set $S$ of
size $s$. Let $C$ be the collection of words formed by all the
columns of the array. Now we prove $C$ is $c$-frameproof for any
$c$ such that $l>c(t-1)$. Let $P$ be any subset
of $C$ with $|P|\leq c$. For any vector $x\in desc(P)\cap C$, each
component of $x$ must agree with the corresponding component of
one of the codewords in $P$. Since $|P|\leq c$, there is a
codeword $y\in P$ that agrees $x$ in at least $t$ positions. Thus
$x=y$ from the definition of orthogonal array.
\end{proof}
\vskip 10pt

Let $C$ be a frameproof code of length $l$ over $S$. Denote the
symmetric group on $S$ by $Sym(S)$.  For each $i\in [l]$,
$\sigma\in Sym(S)$ and for each codeword $b=(b_1,b_2,\ldots,b_l)$,
define
$b(\sigma,i)=(b_1,\ldots,b_{i-1},\sigma(b_i),b_{i+1},\ldots,b_l)$.
Finally, define $C(\sigma,i)=\{b(\sigma,i):b\in C\}$. It is
natural to obtain the following result.

\vskip 10pt
\begin{lemma}
\label{permutation} If $C$ is a $c$-frameproof code of length
$l$ over $S$, then $C(\sigma,i)$ is a $c$-frameproof code for each $i\in [l]$ and  $\sigma\in Sym(S)$.
\end{lemma}
\begin{proof} The proof proceeds by contradiction. Assume that $C(\sigma,i)$ is not
$c$-frameproof, i.e., there exists a codeword $b\in C$ and a set
$P\subset C$ of cardinality $c$, such that $b(\sigma,i)\in
desc(P(\sigma,i))\cap C(\sigma,i)$ but $b(\sigma,i)\notin P(\sigma,i)$. By the
definition of descendant, for each $k\in [l]\setminus \{i\}$,
there exists $y\in P$ such that $b_k=y_k$. For $k=i$, there exists
$y\in P(\sigma,i)$ such that $\sigma(b_i)=\sigma(y_i)$, hence
$b_i=y_i$ because $\sigma$ is a permutation. Thus $b\in desc(P)$ but
$b\notin P$, which is a contradiction with the fact that $C$ is
$c$-frameproof.
\end{proof}
\vskip 10pt

 Let $S$ be a set of
size $s$ containing $\infty$. Suppose there exists an $OA(t,l,s)$ over $S$ which is an
$l\times s^t$ array. Denote the column vectors by $B_i$,
$i=0,1,\ldots,s^t-1$. By the definition of orthogonal array and
Lemma~\ref{permutation}, we can assume that $B_0$ is the all
$\infty$ vector. For each $i\in[s^t-1]$, $B_i$ contains at most
$t-1$ components with $\infty$. Furthermore, a vector $B_i$ is
uniquely determined by specifying $t$ of its components. By
Lemma~\ref{oa2code}, $B_i$, $i\in[s^t-1]$, form an $s$-ary length
$l$ $c$-frameproof code of cardinality $s^t-1$ satisfying Property
$P(t)$, where $c$ is any integer such that $l=c(t-1)+r$ for some
$r\in [c]$. Hence, we have the following construction by
Lemma~\ref{p2code}.

\vskip 10pt
\begin{lemma}
\label{oa2code1} Let $m$ be a prime power and $l,s,t$ be
positive integers such that $m\geq l-1$ and $2t-1\leq l$. Define
$q=(s-1)m+1$. If there exists an $OA(t,l,s)$, then there exists a
$q$-ary length $l$ $c$-frameproof code of cardinality
$\frac{(s^t-1)}{(s-1)^t}(q-1)^t$, where $c\geq t$ is any integer
such that $l=c(t-1)+r$ for some $r\in \{t,t+1,\ldots,c\}$.
\end{lemma}
\vskip 10pt

Applying Lemma~\ref{oa2code1} with the existence of $OA(2,s+1,s)$
for any prime power $s$, we show the following result.

\vskip 10pt
\begin{corollary}
\label{oa} Let $c\geq 2$ be an integer such that $c+1$ is a prime
power, and let $m\geq c+1$ be any prime power.  Then there exists a
$q$-ary $c$-frameproof code of length $c+2$ with cardinality
$\frac{c+2}{c}(q-1)^2$, where $q=cm+1$.
\end{corollary}
\begin{proof} Let $l=c+2$, $s=c+1$
and $t=r=2$, then there exists an $OA(2,l,s)$. By
Lemma~\ref{oa2code1}, there exists a $q$-ary $c$-frameproof code
of length $l$ with cardinality
$\frac{(s^2-1)}{(s-1)^2}(q-1)^2=\frac{(s+1)}{(s-1)}(q-1)^2=\frac{c+2}{c}(q-1)^2$.
\end{proof}
\vskip 10pt

Corollary~\ref{lowerbound1} and Corollary~\ref{oa} combine to
determine the values for $R_{c,c+2}$.

\vskip 10pt
\begin{theorem}
\label{r} Let $c\geq 2$ be an integer such that $c+1$ is a prime
power, then $R_{c,c+2}=(c+2)/c$.
\end{theorem}
\begin{proof} By Corollary~\ref{lowerbound1}, we have $R_{c,c+2}\leq (c+2)/c$. It remains to show that $R_{c,c+2}\geq (c+2)/c$. For a given value of $q$, let $q_l$ be the largest
prime power such that $cq_l+1\leq q$, and let $q_u$ be the
smallest integer such that $cq_u+1\geq q$. That is $q_l$ is the
largest prime power such that $q_l\leq q_u$. By the prime number
theorem, $q_l/q_u=1-o(1)$. By Corollary~\ref{oa}, we have
$M_{c,c+2}(cq_l+1)\geq \frac{c+2}{c}(cq_l)^2$. Hence
\begin{equation*}
\begin{split}
M_{c,c+2}(q)/q^2&\geq M_{c,c+2}(cq_l+1)/q^2\\&\geq
\frac{c+2}{c}(cq_l)^2/q^2\\&\geq
\frac{c+2}{c}(cq_l)^2/(cq_u+1)^2\\&=\frac{c+2}{c}\cdot(\frac{q_l}{q_u+1/c})^2,
\end{split}
\end{equation*}
which shows $R_{c,c+2}\geq (c+2)/c$. This completes the proof.
\end{proof}
\vskip 10pt

\section{Conclusion}

Determining the largest cardinality of a $q$-ary $c$-frameproof
code of length $l$,  $M_{c,l}(q)$ is a difficult problem for
general $c,l,q$. In this paper, we show that the leading term of
the upper bound for $M_{c,l}(q)$ in Theorem~\ref{lowerbound},
proposed by Blackburn \cite{Blackburn2003SIAM}, is tight when $c+1$ is a prime power and
$l=c+2$, by
constructing corresponding frameproof codes of sufficiently large cardinality.

\section*{Acknowledgments}

The authors thank the anonymous reviewers for
their constructive comments and suggestions that greatly improved
the readability of this article. 
The authors express their gratitude to Professor Ying Miao for
kindly mentioning this topic to them.

\vskip 10pt

\begin{IEEEbiographynophoto}{Yeow~Meng~Chee}  received the B.Math. degree in computer science and
combinatorics and optimization and the M.Math. and Ph.D. degrees in computer science,
from the University of Waterloo, Waterloo, ON, Canada, in 1988, 1989, and 1996, respectively.

Currently, he is an Associate Professor at 
the Division of Mathematical Sciences, School of Physical
and Mathematical Sciences, Nanyang Technological University, Singapore.
Prior to this, he was Program Director of Interactive Digital Media R\&D in the
Media Development Authority of Singapore,
Postdoctoral Fellow at the University of Waterloo and
IBM's Z{\"u}rich Research Laboratory, General Manager of the Singapore Computer Emergency
Response Team, and 
Deputy Director of Strategic Programs at the Infocomm Development Authority, Singapore.
His research interest lies in the interplay between combinatorics and computer science/engineering,
particularly combinatorial design theory, coding theory, extremal set systems,
and electronic design automation.

\end{IEEEbiographynophoto}

\begin{IEEEbiographynophoto}{Xiande~Zhang} received the Ph.D.
degree in mathematics from Zhejiang University, Hangzhou,
Zhejiang, P. R. China in 2009. During 2009--2011, she held a
postdoctoral position with  Mathematical Sciences, Nanyang
Technological University, Singapore. She is now a research
fellow with School of Mathematical Sciences, Monash University,
Australia. Her research interests include combinatorial design
theory, coding theory, cryptography, and their interactions.
\end{IEEEbiographynophoto}

\end{document}